\documentclass[12pt]{article}
\usepackage{fancyhdr,graphicx}
\usepackage{amsfonts}
\usepackage{amssymb}
\usepackage{color}
\usepackage{epsf}
\usepackage{amsthm}
\usepackage{newlfont}
\usepackage{amsmath}
\usepackage{bbm}
\usepackage{cite}

\newtheorem{thm}{Theorem}[section]
\newtheorem{lemma}[thm]{Lemma}
\newtheorem{cor}[thm]{Corollary}

\title{Anti-forcing numbers of perfect matchings of graphs}

\author{Hongchuan Lei$^a$, Yeong-Nan Yeh$^a$, Heping Zhang$^b$}

\date{\small $^a$Institute of Mathematics, Academia Sinica, Taipei 10617, Taiwan,\\ E-mail: mayeh@math.sinica.edu.tw\\
\vskip 2mm
$^b$School of Mathematics and Statistics, Lanzhou
University, Lanzhou,  Gansu 730000, China, E-mail:  zhanghp@lzu.edu.cn}

\begin{document}

\maketitle
\begin{abstract}
We  define the anti-forcing number of a perfect matching $M$ of a graph $G$ as the minimal number of edges of $G$ whose deletion results in a subgraph with a unique perfect matching $M$, denoted by $af(G,M)$. The anti-forcing number of a graph proposed by Vuki\v{c}evi\'{c} and Trinajsti\'c in Kekul\'e structures of molecular graphs is in fact the minimum anti-forcing number of perfect matchings.
For plane bipartite graph $G$ with a perfect matching $M$, we obtain a minimax result: $af(G,M)$ equals the maximal number of  $M$-alternating cycles of $G$ where any two either are disjoint or  intersect only at edges in $M$. For a hexagonal system $H$, we show that the maximum anti-forcing number of $H$ equals the Fries number of $H$. As a consequence, we have that the Fries number of $H$ is between the Clar number of $H$ and twice. Further, some extremal graphs are discussed.


\end{abstract}

\section{Introduction}

We only consider finite and simple graphs.  Let $G$ be a graph with vertex set $V(G)$ and edge set $E(G)$. A perfect matching or 1-factor $M$ of a graph $G$ is a set of edges of $G$ such that  each vertex of $G$ is incident with exactly one edge in $M$.

A Kekul\'e structure of some molecular graph (for example, benzenoid and fullerene) coincides with a perfect matching of a graph. Randi\'c and Klein \cite{Randic,Klein} proposed the {\em innate degree of freedom} of a Kekul\'e structure, i.e. the least number of double bonds can determine this entire Kekule structure, nowadays it is called the  forcing number by Harary et al. \cite{Harary}.

A {\em forcing set} $S$ of a perfect matching
$M$ of $G$ is a subset of $M$ such that
$S$ is contained in no other perfect matchings of $G$.
The {\em forcing number} of $M$ is the smallest cardinality over all forcing sets of $M$, denoted by $f(G,M)$.
An edge of $G$ is called a {\em forcing edge} if it is contained in exactly one perfect matching of $G$.
The {\em minimum} (resp. {\em maximum}) \emph{forcing number} of $G$ is the minimum (resp. maximum) value of forcing numbers
of all perfect matchings of $G$, denoted by $f(G)$ (resp. $F(G)$). In general to compute the minimum forcing number of a graph with the maximum degree 3 is an NP-complete problem \cite{Afshani}.

Let $M$ be a perfect matching of a graph $G$.
A cycle $C$ of $G$ is called an {\em $M$-alternating cycle}
if the edges of $C$ appear alternately in $M$ and $E(G)\setminus M$.

\begin{lemma}\cite{Adams,Riddle}\label{forcing} A subset $S\subseteq M$ is a forcing set of $M$ if and only if each $M$-alternating cycle of $G$ contains at least one edge of $S$.
\end{lemma}
For planar bipartite graphs, Pachter and Kim  obtained the following minimax theorem by using Lucchesi and Younger's result in digraphs \cite{LY}.

\begin{thm}\label{cycle}{\rm\cite{Pachter}}
Let $M$ be a perfect matching in a planar bipartite graph $G$. Then $f(G,M)=c(M)$,
where $c(M)$ is the maximum number of disjoint $M$-alternating cycles of $G$.
\end{thm}

A hexagonal system (or benzenoid) is a 2-connected finite plane graph such that
every interior face is a regular hexagon of side length one. It can also
 be formed by a cycle with its interior in the infinite hexagonal lattice on the plane (graphene). A hexagonal system with a  perfect matching is viewed as the carbon-skeleton of a benzenoid hydrocarbon.

Let $H$ be a  hexagonal system with a perfect matching $M$.  A set of {\em disjoint}  $M$-alternating hexagons of $H$ is called an $M$-{\em resonant set}.  A set of  $M$-alternating hexagons of $H$ (the intersection is allowed) is called an $M$-{\em alternating set}.  A  maximum resonant set of $H$ over all perfect matchings is a {\em Clar structure} or {\em Clar set},  and its size is the {\em Clar number} of $H$, denoted by $cl(H)$ (cf. \cite{Hansen2}). A Fries set of $H$ is a maximum alternating set of $H$ over all perfect matchings and the Fries number of $H$, denoted by $Fries(H)$, is  the size of a Fries set of $H$. Both Clar number and Fries number can measure the stability  of polycyclic benzenoid hydrocarbons \cite{Clar, Abeledo}.

\begin{thm}\cite{Xu}\label{clar}
Let $H$ be a  hexagonal system. Then $F(H)=cl(H)$.
\end{thm}

In this paper we consider the anti-forcing number of a graph, which was previously defined by Vuki\v{c}evi\'{c} and Trinajsti\'c \cite{VT2,VT} as the smallest number of edges whose removal results in a subgraph with a single perfect matching (see refs \cite{Che,Dh07,Dh08,LI, ZBV} for some researches on this topic).  By an analogous manner as the forcing number we  define the anti-forcing number, denoted by $af(G,M)$,  of a perfect matching $M$ of a graph $G$ as the minimal number of edges not in $M$ whose removal to fix a single perfect matching $M$ of $G$.  We can see that the anti-forcing number of a graph $G$ is the minimum anti-forcing number of all perfect matchings of $G$.  We also show that the anti-forcing number has a close relation with forcing number: For any perfect matching $M$ of $G$, $f(G,M)\leq af(G,M)\leq (\Delta-1)f(G,M)$, where $\Delta$ denotes the maximum degree of $G$.
For plane bipartite graph $G$,  we obtain a minimax result: For any perfect matching $M$ of $G$, the anti-forcing number of  $M$  equals  the maximal number of  $M$-alternating cycles of $G$ any two members of which intersect only at edges in $M$. 
For a hexagonal system $H$, we show that the maximum anti-forcing number of $H$ equals the Fries number of $H$. As a consequence, we have that the Fries number of $H$ is between the Clar number of $H$ and twice. Discussions for some extremal graphs about the anti-forcing numbers show the anti-forcing number of a graph $G$ with the maximum degree three can achieve the minimum forcing number or twice.

\section{Anti-forcing number of perfect matchings}

An anti-forcing set $S$ of a graph $G$ is a set of edges of $G$ such that $G-S$ has a unique perfect matching. The smallest cardinality of  anti-forcing sets of $G$ is called the {\em anti-forcing number} of $G$ and denoted by $af(G)$.

Given a perfect matching $M$ of a graph $G$. If $C$ is an $M$-alternating cycle of $G$, then the symmetric difference $M\oplus C$ is another perfect matching of $G$.  Here  $C$ may be  viewed as its edge-set. A subset $S\subseteq E(G)\setminus M$ is called an anti-forcing set of $M$ if $G-S$ has a unique perfect matching, that is, $M$.

\begin{lemma}\label{anti}A set $S$ of edges of $G$ not in $M$ is an anti-forcing set of $M$ if and only if $S$ contains at least one edge of every $M$-alternating cycle of $G$.
\end{lemma}

\begin{proof}If $S$ is an anti-forcing set of $M$, then $G-S$ has a unique perfect matching, i.e. $M$. So $G-S$ has no $M$-alternating cycles. Otherwise, if $G-S$ has an $M$-alternating cycle $C$, then the symmetric difference $M\oplus C$ is another perfect matching of $G-S$ different from $M$, a contradiction. Hence each $M$-alternating cycle of $G$ contains at least one edge of $S$.  Conversely, suppose that $S$ contains at least one edge of every $M$-alternating cycle of $G$. That is, $G-S$ has no $M$-alternating cycles, so $G-S$ has a unique perfect matching.
\end{proof}

The smallest cardinality of anti-forcing sets of $M$ is called the anti-forcing number of $M$ and denoted by $af(G,M)$. So we have the following relations between the forcing number and anti-forcing number.

\begin{thm}\label{bound}Let $G$ be a graph with the maximum degree $\Delta$. For any perfect matching $M$ of $G$, we have $$f(G,M)\leq af(G,M)\leq (\Delta-1)f(G,M).$$
\end{thm}

\begin{proof}Given any anti-forcing set $S$ of $M$. For each edge $e$ in $S$, let $e_1$ and $e_2$ be the edges in $M$ adjacent to $e$. All such edges $e$ in $S$ are replaced with one of $e_1$ and $e_2$ to get another set $S'$ of edges in $M$. It is obvious that $|S'|\leq |S|$. Further we claim that $S'$ is a forcing set of $M$. For any $M$-alternating cycle $C$ of $G$, by Lemma \ref{anti} $C$ must contain an edge $e$ in $S$. Then $C$ must pass through both $e_1$ and $e_2$. By the definition for $S'$, $C$ contains at least one edge of $S'$. So Lemma \ref{forcing} implies that $S'$ is a forcing set of $M$. Hence the claim holds. So  $f(G,M)\leq |S'|\leq |S|$, and the first inequality is proved.

Now we consider the second inequality. Let $F$ be a minimum forcing set of $M$. Then $f(G,M)=|F|$. For each edge $e$ in $F$, we choose all the edges not in $M$ incident with one end of $e$. All such edges form a set $F'$ of size no larger than $(\Delta-1)|F|$, which is disjoint with $M$. We claim that $F'$ is an anti-forcing set of $M$. Otherwise, Lemma \ref{anti} implies that $G-F'$ contains an $M$-alternating cycle $C$. Since each edge in $F$ is a pendant edge of $G-F'$, $C$ does not pass through an edge of $F$. This contradicts that $F$ is a forcing set of $M$ by Lemma \ref{forcing}. Hence $af(G,M)\leq |F'|\leq (\Delta-1)|F|$.
\end{proof}

\begin{lemma}$af(G)=\min\{af(G,M): M \mbox{ is a perfect matching of } G\}$.
\end{lemma}

By the definitions the above result is immediate.  Hence we may say, $af(G)$ is the {\em minimum anti-forcing number} of $G$. Whereas, $$Af(G):=\max\{af(G,M): M \mbox{ is a perfect matching of } G\}$$ is  the {\em maximum anti-forcing number} of $G$.

The following is an immediate consequence of Theorem \ref{bound}.
\begin{cor}\label{bounds}Let $G$ be a graph with a perfect matching and the maximum degree $\Delta$. Then
$$f(G)\leq af(G)\leq (\Delta-1)f(G), F(G)\leq Af(G)\leq (\Delta-1)F(G).$$
\end{cor}

Further,
$\text{Spec}_{f}(G):=\{f(G,M): M \mbox{ is a perfect matching of } G\}$ and $
\text{Spec}_{af}(G):=\{af(G,M): M \mbox{ is a perfect matching of } G\}$
are called the {\em forcing spectrum} \cite{Afshani} and the {\em anti-forcing spectrum} of $G$ respectively. For example, $\text{Spec}_{af}(\text{Triphenylene})=\{2,3,4\}$ and $\text{Spec}_{f}(\text{Triphenylene})=\{1,3\}$ (see Fig. \ref{spec}(a)),  $\text{Spec}_{f}(\text{Dodecahedron})=\{3\}$ \cite{Ye}(see Fig. \ref{spec}(b)).  Randi\'{c} and Vuki\v{c}evi\'{c} \cite{c60,c70} computed the distributions of forcing numbers of Kekul\'e structures of $\text C_{60}$ and $\text C_{70}$ respectively.

\begin{figure}
\centering
\includegraphics[scale=0.6]{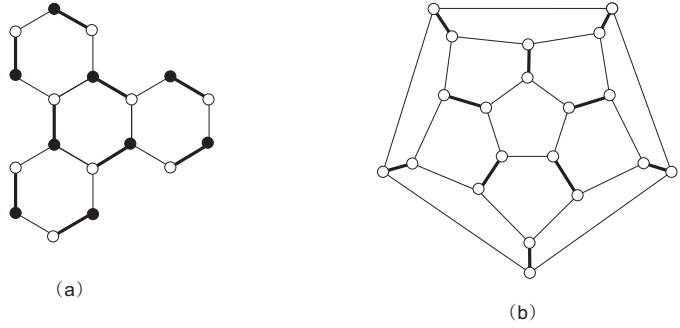}
\caption{(a) Triphenylene, (b) Dodecahedron.}\label{spec}
\end{figure}

For any given graph $G$ with a perfect matching $M$, we now consider the anti-forcing number $af(G,M)$. If $G$ has two $M$-alternating cycles that either are disjoint or intersect only at edges in $M$, then by Lemma \ref{anti} any anti-forcing set of $M$ contains an edge of each one of such $M$-alternating cycles. Thus it naturally motivates us to propose a novel concept: a collection $A$ of $M$-alternating cycles of $G$ is called a {\em compatible $M$-alternating set} if any two members of $A$ either are disjoint or intersect only at edges in $M$. Let $c'(M)$ denote the maximum cardinality of compatible $M$-alternating sets of $G$. By the above discussion we have the following immediate result.

\begin{lemma}\label{ineqality}For any perfect matching $M$ of a graph $G$, we have $af(G,M)\geq c'(M)$.
\end{lemma}

For plane bipartite graphs $G$ we can show that the equality in the above lemma always holds. The vertices of $G$ are colored with white and black such that any pair of adjacent vertices receive different colors. Such two color classes form a bipartition of $G$.

\begin{thm}\label{minimax}Let $G$ be a planar bipartite graph with a perfect matching $M$. Then $$af(G,M)=c'(M).$$
\end{thm}

To obtain  such a minimax result we need a classical result of Lucchesi and Younger \cite{LY} about directed graphs; Its shorter proof was ever given by Lov\'{a}sz \cite{LL}.  Let $D$ be a finite directed graph. A {\em feedback set} of $D$ is a set of arcs that contains at least one arc of each directed cycle of $D$.

\begin{thm}[Lucchesi and Younger]\cite{LY}\label{LY}For a finite planar digraph, a minimum feedback set has cardinality equal to that of a maximum  collection of arc-disjoint directed cycles.
\end{thm}
\noindent{\em Proof of Theorem \ref{minimax}.}  First assign a specific orientation of $G$ concerning $M$ to obtain a digraph $\vec G(M)$: any edge in $M$ is directed from white end to black end, and the  edges not in $M$ are directed from black ends to white ends. Obviously the  $M$-alternating cycles of $G$ corresponds naturally to directed cycles of its orientation. Then contract each edge of $M$ in $\vec G(M)$ to  a vertex (i.e. delete the edge and identify its ends) to get a new digraph, denoted by $\vec G\cdot M$.  We can see that there is a one-to-one correspondence between the  $M$-alternating cycles of $G$ and directed cycles of $\vec G\cdot M$. That is, an $M$-alternating cycle of $G$ becomes a directed cycle $\vec G\cdot M$, and a directed cycle of $\vec G\cdot M$ can produce an $M$-alternating cycle of $G$ when each vertex is restored to an edge of $M$. So by Lemma \ref{anti} a subset $S\subseteq E(G)\setminus M$ is an anti-forcing set of $M$ if and only if $S$ is a feedback set of $\vec G\cdot M$. Hence $af(G,M)$ equals the smallest cardinality of feedback sets of $\vec G\cdot M$. On the other hand, a compatible $M$-alternating set of $G$ corresponds to a set of arc-disjoint directed cycles of  $\vec G\cdot M$. That implies that $c'(M)$ equals the maximum number   of arc-disjoint  directed cycles of $\vec G\cdot M$. Note that $\vec G\cdot M$ is a planar digraph. So  Theorem  \ref{LY} implies $af(G,M)=c'(M)$. \hfill$\square$

\vskip 0.2cm

However,  the equality  in Lemma \ref{ineqality} does not necessarily hold in general. A counterexample is dodecahedron (see Fig. \ref{spec}(b)); For this specific perfect matching marked by bold lines, it can be confirmed that there are at most three compatible alternating cycles, but its anti-forcing number is at least four.

%


%

\section{Maximum anti-forcing number}

In this section we restrict our consideration to a hexagonal system $H$ with a perfect matching $M$. Without loss of generality, $H$ is placed in the plane such that an edge-direction is vertical and the peaks (i.e. those vertices of $H$ that just have two low neighbors, but no high neighbors) are black. An $M$-alternating cycle $C$  of $H$ is said to be {\em proper} (resp. {\em improper}) if each edge of $C$ in $M$ goes from white end to black end (resp. from black end to white end) along the clockwise direction of $C$. The boundary of $H$ means the boundary of the outer face. An edge on the boundary is a boundary edge.

The following main result shows that the maximum anti-forcing number equals the Fries number in a hexagonal system.

\begin{thm}\label{fries}Let $H$ be a hexagonal system with a perfect matching. Then $Af(H)=Fries(H)$.
\end{thm}

\begin{proof}Since any Fries set of $H$ is a compatible $M$-alternating set $A$ for some perfect matching $M$ of $H$, we have that $Af(H)\geq Fries(H)$ from Theorem \ref{minimax}. So we now prove that $Af(H)\leq Fries(H)$. It suffices to prove that for a  compatible alternating set $A$ of $H$ with $|A|=Af(H)$, we can find a { Fries set} $F$ of $H$ such that $|A|\leq |F|$.

Given any compatible $M$-alternating set $A$ of $H$ with a perfect matching $M$. Two cycles $C_1$ and $C_2$ in $A$  are {\em crossing} if they share an edge $e$ in $M$ and the four edges adjacent to $e$ alternate in $C_1$ and $C_2$ (i.e. $C_1$ enters into $C_2$ from one side and leaves from the other side via $e$). Such an edge $e$ is said to be a crossing. For example, see Fig. 2.  We say $A$ is {\em non-crossing} if any two cycles in $A$ are not crossing.

\begin{figure}
\centering
\includegraphics[scale=0.75]{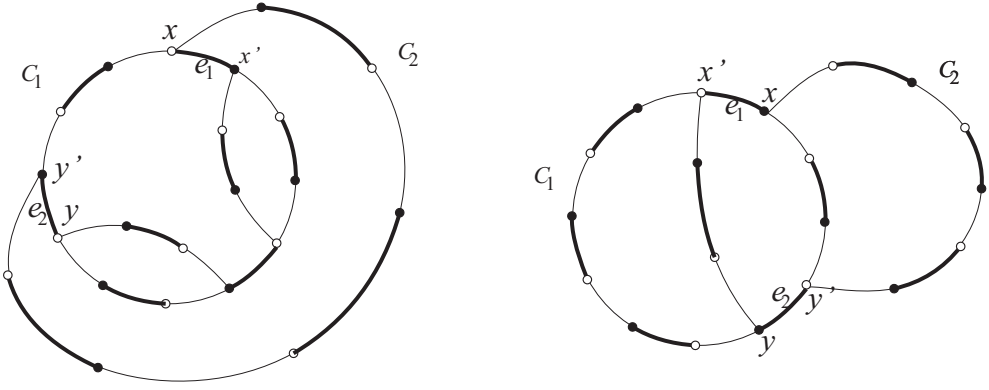}
\caption{Two ways of crossing $M$-alternating cycles $C_1$ and $C_2$ (bold lines are edges  in $M$).} \label{cross}
\end{figure}

\noindent{\bf Claim 1.} For  any compatible $M$-alternating set $A$ of $H$, we can find the corresponding non-crossing compatible $M$-alternating set $A'$ of $H$ such that $|A'|=|A|$.

\begin{proof} Suppose $A$ has a pair of crossing members $C_1$ and $C_2$. In fact $C_1$ and $C_2$ have even number of crossings. Let $e_1$ and $e_2$ be two consecutive crossings, which   are edges in $M$.  So we may suppose along the counterclockwise direction  $C_2$ from edge $e_1=xx'$ enters into the interior of $C_1$, then reaches the crossing $e_2=yy'$. Note that $x$ is the first vertex of $C_2$ entering in $C_1$ and $y'$ the first vertex of $C_2$ leaving from $C_1$ after $x$.  For convenience, if a cycle $C$ in $H$ has two vertices $s$ and $t$, we always denote by $C(s,t)$ the path from $s$ to $t$ along $C$ clockwise. If $C_1$ is a proper $M$-alternating cycle and $C_2$ is an improper $M$-alternating cycle,   let  $C_1':=C_1(y,x')+C_2(y,x')$ and $C_2':=C_1(x',y)+C_2(x',y)$ (see Fig. 2(left)).   If $C_1$ and $C_2$ both are proper (resp. improper) $M$-alternating cycles,   let  $C_1':=C_1(y',x)+C_2(x,y')$ and $C_2':=C_1(x,y')+C_2(y',x)$ (see Fig.  2(right)). In all such cases $C_1$ and $C_2$ in $A$ can be replaced with $C_1'$ and $C_2'$ to get a new compatible $M$-alternating set of $H$ and such a pair of crossings $e_1$ and $e_2$ disappeared. Since such a change cannot produce any new crossings, by repeating the above process we finally get a compatible $M$-alternating set $A'$ of $H$ that is non-crossing.  It is obvious that $|A'|=|A|$.  \end{proof}

For a cycle $C$ of $H$, let $h(C)$ denote the number of hexagons in the interior of $C$.  By Claim 1 we can choose  a perfect matching $M$ of $H$ and a maximum compatible $M$-alternating set $A$  satisfying that (i) $|A|=Af(H)$ and (ii) $A$ is non-crossing,   and $h(A):=\sum_{C\in A}{h(C)}$ is as minimal as possible subject to (i) and (ii). We call $h(A)$ the $h$-{\em index} of $A$.

By the above choice we know that for any two cycles in $A$ their interiors either are disjoint or one contains the other one. Hence the cycles in $A$ form a {\em poset} according to the containment relation of their interiors. Since each $M$-alternating cycle has an $M$-alternating hexagon in its interior (cf. \cite{Zhang}), we immediately obtain the following claim.

\noindent{\bf Claim 2.} Every minimal member of $A$ is a hexagon.

It suffices to prove that all members of $A$ are hexagons. Suppose to the contrary that  $A$ has at least one non-hexagon member.   Let  $C$ be  a minimal non-hexagon member in $A$. Then $C$ is an $M$-alternating cycle. We consider a new hexagonal system $H'$ formed by $C$ and its interior as a subgraph of $H$. Without loss of generality, suppose that $C$ is a proper $M$-alternating cycle (otherwise, analogous arguments are implemented on right-top corner of $H'$).  So we can find  a substructure of $H'$ in its left-top corner as follows.

\begin{figure}
\centering
\includegraphics[scale=0.9]{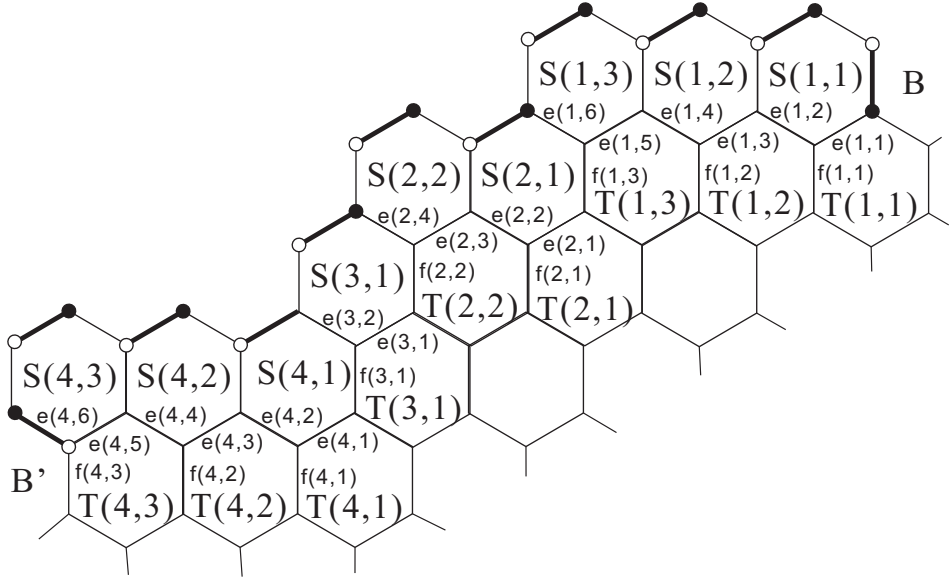}
\caption{Illustration for the proof of Claim 3 (bold lines are edges  in $M$, $m=4,n(1)=3,(n(2)=2,n(3)=1,n(4)=3)$.}\label{claim3}
\end{figure}

We follow the notations of Zheng and Chen \cite{ZC}.  Let $S(i,j), 1\le i\le m$ and $1\le j\le n(i)$, be a series of hexagons on the boundary of $H'$ as Fig. 3 that form a hexagonal chain and satisfy that neither $B$ nor $B'$ is contained in $H'$. We denote edges, if any, by $e(i,k), 1\le i\le m$ and $1\le k\le 2n(i)$, and by $f(i,j),  1\le i\le m$ and $1\le j\le n(i)$; and denote the hexagons (not necessarily contained in $H'$) with both edges $f(i,j)$  and $e(i, 2j-1)$, by $T(i,j), 1\le i\le m$ and $1\le j\le n(i)$ (see Fig. 3).

\noindent{\bf Claim 3.} (a) $n(1)=1$, and $m\geq 2$,\\
(b)  $n(i)=1$ or 2 for all $1\leq i\leq m$,  \\
(c)  for all $1\leq i\leq m$, $f(i,n(i))\in M$, and \\
(d)  if $n(i)=2$, $2\le i\le m$, then  $S(i,1)\in A$.

\begin{proof}We now prove the claim by induction on $i$. We first consider $i=1$.  If $e(1,2)\in M$, then $S(1,1)$ is a proper $M$-alternating hexagon. So $C$ in $A$ can be replaced with $S(1,1)$ to produce a new compatible  $M$-alternating set $A'$. That is, $A':=(A\cup S(1,1))-\{C\}$,  but $|h(A')|<|h(A)|$, a contradiction. So $e(1,2)\notin M$, which implies that  $f(1,1)\in M$ and all edges $e(1,3),e(1,5),\ldots, e(1, 2n(1)-1)$  belong to $M$. Hence $S(2,1)$ is a hexagon of $H'$ and $m\geq 2$.  If $n(1)\geq 2$, since the boundary $C$ of $H'$ is a proper $M$-alternating cycle,  none of the edges $e(1,2),e(1,3),\ldots, e((1,2n(1))$  is a boundary edge of $H'$.
In this case the cycle $C$ can be replaced with $C\oplus S(1,n(1)$ to get another compatible $M$-alternating set with less index $h$-index than  $A$, also a contradiction. Hence $n(1)=1$.  So the claim holds for $i=1$.

  Suppose $1\leq i< m$ and  Claim 3 holds for any integer $1\leq i'\leq i$. We want to show that it holds for $i+1$.   There are  two cases to be considered.

{\bf Case 1.}  $n(i)=1$. Suppose that  $n(i+1)\geq 3$. If $e(i+1,2)\notin M$, then  $e(i+1,3), e(i+1,5),\ldots, e(i+1, 2n(i)-1)$ all belong to $M$. By an analogous argument as above, we have that   $T(i+1,2),\ldots, T(i+1,n(i+1)), S(i+2,1)$ are hexagons of $H'$,  and $C$ can be replaced with $C\oplus S(i+1,n(i+1))$ to get another $M$-compatible alternating set with less  $h$-index than  $h$, also a contradiction.  Hence $e(i+1,2)\in M$. By the induction hypothesis we have  $f(i,1)\in M$,  and  $S(i+1,1)$ is an $M$-alternating hexagon. If $e(i+1,4)\notin M$, the similar contradiction occurs. So  $e(i+1,4)\in M$. We can see that none of members of $A$ but $C$ intersect $S(i+1,1)$. Then $(A\cup \{S(i,1),S(i+1,1),S(i+2,2)\})-\{C\}$ is a compatible $M\oplus S(i+1,1)$-alternating set, which is larger than $A$, contradicting the choice of $A$. Hence $n(i+1)\le 2$.  If $n(i+1)=1$, then $f(i+1,1)\in M$. Otherwise, $C$ in $A$  would be replaced with $S(i+1,1)$ to obtain a similar contradiction. If $n(i+1)=2$,  by the similar arguments we have that $e(i+1,2)\in M$ and $f(i+1,2)\in M$. So $S(i+1,1)\in A$.

{\bf Case 2.} $n(i)=2$. Choose an integer $i_0$ with $1\leq i_0 <i$ such that $n(i_0)=1$, and $n(i_0+1)=n(i_0+2)=\cdots =n(i)=2$.  By the induction hypothesis, we have that the right vertical edge of hexagon $S(i_0,1)$ belongs to $M$, the hexagons $S(i_0+1,1), S(i_0+2,1),\ldots, S(i,1)$ are all proper $M$-alternating hexagons, which all belong to $A$, and $f(i,2)\in M$.  If $e(i+1, 2)\notin M$, then $f(i+1,1)\in M$. We have that $n(i+1)=1$; otherwise, $n(i+1)\geq 2$ and $C$ would be replaced with $C\oplus S(i+1,n(i+1))$ to get another $M$-compatible alternating set with less $h$-index  than  $A$, also a contradiction.  So suppose that  $e(i+1, 2)\in M$. Then $S(i+1,1)$ is a proper $M$-alternating hexagon. We claim that $n(i+1)=2$ and $f(i+1,n(i+1))\in M$.  If $n(i+1)=1$, then $e(i+1, 2)$ belongs to $C$.  So $C$ can be replaced with  $S(i+1,1)$ also to get a contradiction.  Hence  $n(i+1)\geq 2$.  Suppose $e(i+1,4)\in M$. Let $M'=M\oplus S(i+1,1)\oplus S(i,1)\oplus\cdots \oplus S(i_0+1,1)$ .  Then $M'$ is a perfect matching of $H$ so that $S(i+1,2),  S(i+1,1), S(i,2),  S(i,1), \ldots, S(i_0+1,2),  S(i_0+1,1),S(i_0,1)$ are $M$-alternating hexagons. Let $A':=(A\cup\{S(i+1,2),   S(i,2),   \ldots, S(i_0+1,2) ,S(i_0,1)\})-\{C,T(i,2),\ldots,T(i_0+1,2)\}$.  Then $A'$ is a compatible $M'$-alternating set of $H$ with $|A|< |A'|$, contradicting the choice for $A$. Hence $e(i+1,4)\notin M$ and $f(i+1,2)\in M$. If $n(i+1)\geq 3$, then $e(i+1,5),  e(i+1,7),\ldots, e(i+1,2n(i+1)-1)$ all belong to $M$, so $C$ can be replaced with $C\oplus S(i+1,n(i+1))$ to get a similar contradiction. Hence $n(i+1)=2$ and the claim holds. Further we have that $S(i+1,1)\in A$

 Now we have completed the proof of Claim 3.  \end{proof}

By Claim 3 we have that  $f(m,n(m))\in M$. That implies that $e(m, 2n(m))\notin M$. So $S(m+1,1)$ exists in $H'$, a contradiction. Hence each member of $A$ is a hexagon.
\end{proof}

Combining Theorems \ref{clar} and \ref{fries} with Corollary \ref{bounds}, we immediately obtain the following relations between the Clar number and Fries number.

\begin{cor}Let $H$ be a hexagonal system. Then $cl(H)\leq Fries(H)\leq 2cl(H)$.
\end{cor}

\section{Some extremal classes}

\subsection{All-kink catahexes}
Let $H$ be a hexagonal system. The inner dual $H^*$ of $H$ is a plane graph: the center of each hexagon $h$ of $H$ is placed a vertex $h^*$ of $H^*$, and if two hexagons of $H$ share an edge, then the corresponding  vertices are joined by an edge. $H$  is called catacondensed if its inner dual is a tree. Further $H$ is called {\em all-kink catahex} \cite{Harary} if it is catacondensed and no two hexagons  share a pair of parallel edges of a hexagons. The following result due to Harary et al. gives a characterization for a hexagonal system to have the Fries number (or the maximum anti-forcing number) achieving the number of hexagons.

\begin{thm}\cite{Harary}\label{Harary} For a hexagonal system $H$ with $n$ hexagons, $Fries(H)\leq n$, and equality holds  if and only if $H$ is an all-kink catahex.
\end{thm}

An  independent (or stable) set of a graph $G$ is a set of  vertices  no two of which are  adjacent. The independence number of $G$, denoted by $\alpha(G)$, is the largest cardinality of independent sets of $G$.

\begin{thm}For an all-kink catahex $H$, $Af(H)=2F(H)$ if and only if the inner dual $H^*$ has a perfect matching.
\end{thm}

\begin{proof}By Theorem \ref{Harary}, $Af(H)$ equals the number $n$ of vertices of $H^*$. Note that any set of disjoint hexagons of $H$ is a resonant set. By Theorem \ref{clar}, $F(H)=cl(H)=\alpha(H^*)$. Since $H^*$ is a bipartite graph, $\nu(H^*)+\alpha(H^*)=n$, where $\nu(H^*)$ denotes the matching number of $H^*$, the size of a maximum matching of $H^*$. So this equality implies the result.\end{proof}

For a hexagonal system $H$ with a perfect matching $M$, let $fries(M)$ be the number of $M$-alternating hexagons of $H$. Then $Fries(H)$ is the maximal value of $fries(M)$ over all perfect matchings. The minimal value of $fries(M)$ over all perfect matchings $M$ is called the {\em minimum fries number}, denoted by $fries(H)$. For an all-kink catahex, each hexagon has two choices for three disjoint edges, and just one's edges can be glued with other hexagons, so these three edges are called {\em fusing edges}. If a fusing edge is on the boundary, then an additive hexagon is glued along it to get a larger all-kink catahex.

 A {\em dominating set} of a graph $G$ is a set $S$ of vertices of $G$ such that every vertex not in $S$ has a neighbor in $S$.
An independent dominating set of $G$ is a set of vertices of $G$ that is
both dominating and independent in $G$ \cite{GH}. The independent domination number of $G$, denoted by $i(G)$, is the minimum size of
 independent dominating sets of $G$. (For a survey on independent domination, see \cite{GH})

\begin{thm}For an all-kink catahex $H$, $f(H)=i(H^*)=fries(H)$.
\end{thm}

\begin{proof} For any perfect matching $M$ of $H$, by Theorem \ref{cycle} we have that $f(H,M)=c(M)$. Note that  $H$ has no interior vertices. Since each $M$-alternating cycle of $H$ contains an $M$-alternating hexagon in its interior, $c(M)$ equals the maximum number of disjoint $M$-alternating hexagons of $H$.  It is obvious that for a hexagon of $H$ a non-fusing edge belongs  to $M$ if and only if the three non-fusing edges belong to $M$.

Choose a perfect matching $M$ of $H$ such that $f(H)=f(H,M)$.
Let $S$ be a maximum set of disjoint $M$-alternating hexagons of $H$ and $S^*:=\{h^*:h\in S\}$. Then $f(H)=|S^*|$. We claim that  $S^*$ is an independent dominating set of $H^*$. Let $h$ be any hexagon  of $H$ not in $S$.  If some hexagon $h'$ of $H$ adjacent to $h$ has the three non-fusing edges in $M$, then $h'\in S$. Otherwise, $h$ is an $M$-alternating hexagon. Since $h\notin S$ and $S$ is maximum, some hexagon of $H$ adjacent to $h$ must belong to $S$. So the claim holds, and $f(H)\geq i(H^*)$. Conversely, given a minimum independent dominating set $S^*$ of $H^*$. Construct a perfect matching $M_0$ of $H$ as follows. The three non-fusing edges of each hexagon in $S$ are chosen as edges of $M_0$. For any hexagon of $H$ not in $S$,  a fusing edge that  is  a boundary edge or   shared by the other hexagon not in $S$ is also an edge of $M_0$. So we can see that $M_0$ is a perfect matching of $H$ and any hexagon of $H$ not in $S$ is not $M_0$-alternating. Hence $S$ is the maximum set of $M_0$-alternating hexagons of $H$. So $i(H^*)=f(H,M_0)\geq f(H)$. Hence $i(H^*)= f(H)$.

\begin{figure}
\centering
\includegraphics[scale=0.5]{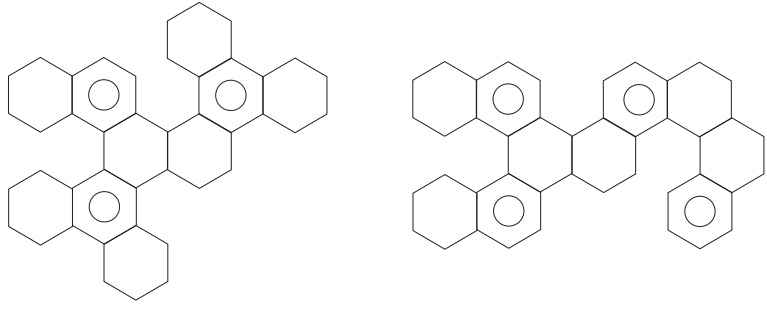}
\caption{All-kink catahexes with the minimum forcing numbers 3 and 4}\label{cata}
\end{figure}

According to the above construction,  $S$ is the set of all $M_0$-alternating hexagons of $H$. Hence $f(H)=h(M_0)\geq fries(H)$. On the other hand, for any perfect matching $M$ of $H$,  $c(M)\leq fries(M)$, and thus $f(H)\leq fries(H)$.
Both inequalities  imply the second equality.
\end{proof}

Beyer et al. \cite{BPHM} observed an algorithm of linear time to compute the independent domination number of a tree. So the minimum forcing number of  all-kink catahexes can be computed in linear time. For example, Fig. \ref{cata} gives the minimum forcing numbers of two all-kink catahexes. But the anti-forcing number of an all-kink catahex may be larger than its minimum forcing number; for example, the triphenylene has the minimum forcing number 1 and the anti-forcing number 2 (see Fig. \ref{spec}(a)).

\subsection{$af(H)=1,2$}

Li \cite{LI} gave the structure of hexagonal systems with an anti-forcing edge (i.e.  an edge that itself forms an anti-forcing set).  For integers $n_1\geq n_2\geq \cdots \geq n_k$, let $H(n_1,n_2,\ldots,n_k)$ be a hexagonal system with $k$ horizontal rows of $n_1\geq n_2\geq \cdots \geq n_k$  hexagons and last hexagon of each row being immediately  below and to the right of the last one in the previous row, and we call it {\em truncated parallelogram} \cite{Cyvin}; For example, See Fig. \ref{para}. In particular, $H(r,r,\ldots,r)$ with $k\geq 2$ and $r\geq 2$ and $H(r)$ with $r\geq 2$ are parallelogram and linear chain respectively. Note that a truncated parallelogram can be placed and represented in other ways.

\begin{figure}
\centering
\includegraphics[scale=0.6]{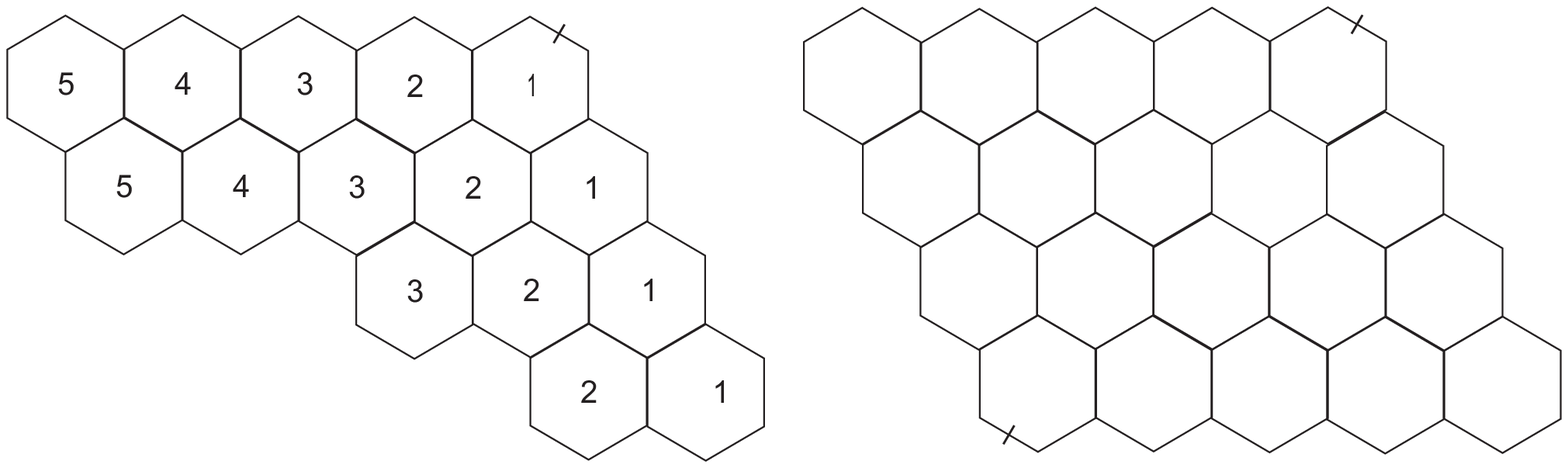}
\caption{Truncated parallelograms $H(5,5,3,2)$ and $H(5,5,5,5)$} with anti-forcing edges marked by  short lines.\label{para}
\end{figure}

\begin{thm}\cite{LI}\label{LI} Let $H$ be a hexagonal system. Then $af(H)=1$ if and only if $H$ is a truncated parallelogram.
\end{thm}

Precisely, a single hexagon has six anti-forcing edges, a linear chain has four anti-forcing edges, and a parallelogram has two anti-forcing edges. A true truncated parallelogram has just one anti-forcing edge (see Fig. \ref{para}). In the following we will give a construction for  hexagonal systems with the anti-forcing number 2.

Some necessary preliminary is needed. Let $G$ be a connected plane bipartite graph. An edge of $G$ is said to be {\em fixed single} (resp. {\em double}) if it belongs to no (resp. all ) perfect matchings of $G$. $G$  is {\em normal} or {\em elementary} if $G$ has no fixed single edges. The non-fixed edges of $G$ form a subgraph whose components are normal  and thus 2-connected graphs, which are  called {\em normal components} of $G$. Further, a normal component of $G$ is called a {\em normal block} if it is formed by a cycle of $G$ with its interior. A  pendant vertex of a graph is a vertex of  degree one, and its incident edge is a pendant  edge.

\begin{lemma}\cite{LP}\label{unique} If a bipartite graph has a unique perfect matching, then it has a pair of pendant vertices with different colors.
\end{lemma}

\begin{lemma}\cite{SLZZ}\label{component} Let $H$ be a connected plane bipartite graph with a perfect matching. If all pendant vertices of $G$ are of the same color and lie on the boundary, then $G$ has at least one normal block. If $G$ has a fixed single edge and $\delta(G)\geq 2$, then $G$ has at least two normal blocks.
\end{lemma}

The following result was first pointed out by Sachs and can be extended to bipartite graphs.

\begin{lemma}\cite{Sachs}\label{constant} Let $H$ be a hexagonal system with a perfect
matching. Let $E=\{e_1, e_2, \ldots, e_r\}$ be a set of parallel edges of $H$ such that
$e_i$ and $e_{i+1}$  belong to the same hexagon and the $e_1$ and $e_r$ are  boundary edges. Then $E$ is an edge-cut of $H$ and $|E\cap M|$ is invariant for all perfect matchings $M$ of $H$.
\end{lemma}

\begin{thm}\label{af2}Let $H$ be a hexagonal system with a fixed single edge. Then $af(H)=2$ if and only if $H$ has exactly two normal components, which are both truncated parallelograms.
\end{thm}
\begin{proof}By Lemma \ref{component}  $H$ has at least two normal components. Such normal component is a hexagonal system with the anti-forcing number at least one. Note that the anti-forcing number of $H$ equals the sum of the anti-forcing numbers of such normal components. Hence $af(H)=2$ if and only if $H$ has exactly two normal components, which are truncated parallelograms by Theorem \ref{LI}.
\end{proof}

\begin{thm}\label{af2}Let $H$ be a normal hexagonal system. Then $af(H)=2$ if and only if $H$  is not truncated parallelogram and $H$ can be obtained by gluing two truncated parallelograms $T_1$ and $T_2$ along their boundary parts as a fused path $P$ of odd length such that\\
(i) an anti-forcing edge of $T_1$ remains on the boundary,\\
(ii) the hexagons of each $T_i$ with an edge of $P$ form a linear chain or a chain with one kink (i.e. the inner dual is a path with exactly one turning vertex), and\\
(iii) when the fused path $P$ passes through edge $b$ (or $a$) of $T_1$, the hexagons of $T_1$ (resp. $T_2$) with an edge of $P$ form a linear chain that is the last (or first ) row of $T_1$ (resp. a chain with one kink). (see Fig. \ref{construction})
\end{thm}

\begin{figure}
\centering
\includegraphics[scale=0.6]{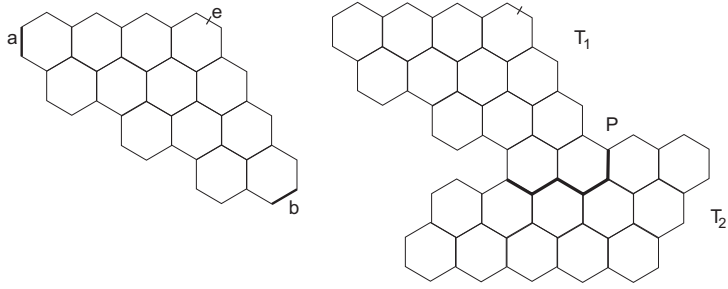}
\caption{(a) A truncated parallelogram with specific edges $a$ and $b$, (b) Gluing two truncated parallelograms.}\label{construction}
\end{figure}

\begin{proof}Suppose that $af(H)=2$. Then  $H$ has  distinct edges $e$ and $e'$ such that $H':=H-e-e'$ has a unique perfect matching $M$. So by  Lemma \ref{unique} $H'$ has two pendant vertices with different colors. Then one of $e$ and $e'$, say $e$,  must be a boundary edge of $H$; otherwise $H'$ has at most one pendant vertex, a contradiction.

\noindent {\bf Claim 1}. $e$ has at least one end with degree two in $H$.

Otherwise, suppose that $e$ has both ends with degree three. Then $H-e$ has the minimum degree two. If $H-e$ is 2-connected, it must be a hexagonal system other than truncated parallelogram, contradicting that $H-e$ has an anti-forcing edge $e'$. If $H-e$ has a cut edge, by Lemma \ref{component} $H-e$ has at least two normal components. So $af(H-e)\geq 2$, also a contradiction, and Claim 1 holds.

So $H-e$ has a pendant vertex $x$. The edge $e_0$ between $x$ and  its neighbor belongs to all perfect matchings of $H-e$, and is thus anti-forced by $e$. Deleting the ends of this edge and incident edges, any pendant  edges of the resulting graph also belong to all perfect matchings of $H-e$, such pendant edges are anti-forced by $e$. Repeating the above process, until to get a graph without pendant vertices, denoted by $H\ominus e$.

\noindent {\bf Claim 2}. $H\ominus e$ is a truncated parallelogram with an anti-forcing edge $e'$.

If $H\ominus e$ is empty, then $e$ is an anti-forcing edge of $H$, a contradiction. Otherwise,  $H\ominus e$ has a perfect matching and  the minimum degree two. Note that the interior faces of $H\ominus e$ are hexagons. By the similar arguments as the proof of Claim 1,  we have that $H\ominus e$ is a hexagonal system with an anti-forcing edge $e'$. Hence Claim 2 holds.

Without loss of generality, suppose that  edge $e$ is from the left-up end $x$ to the right-low end. Then $e_0$ is a slant edge. Let $s$ be the hexagon with edge $e$,  $f_0$  the vertical edge of  $s$ adjacent to $e$, $d_0$ the other edge of $s$ parallel to $e$. From the center $O$ of $s$ draw a ray perpendicular to and away from $f_0$ (resp. $e_0$) intersecting a boundary edge $a$ at $A$ (resp. edge $b$ at $B$) such that $OA$ (resp. $OB$) only passes through hexagons of $H$. Let  $H_0$ and $H_0'$ be the linear chains of $H$ consisting of hexagons intersecting $OA$ and $OB$; See Fig. \ref{Case1}. By the similar reasons as Claim 1, we have the following claim.

\begin{figure}
\centering
\includegraphics[scale=0.65]{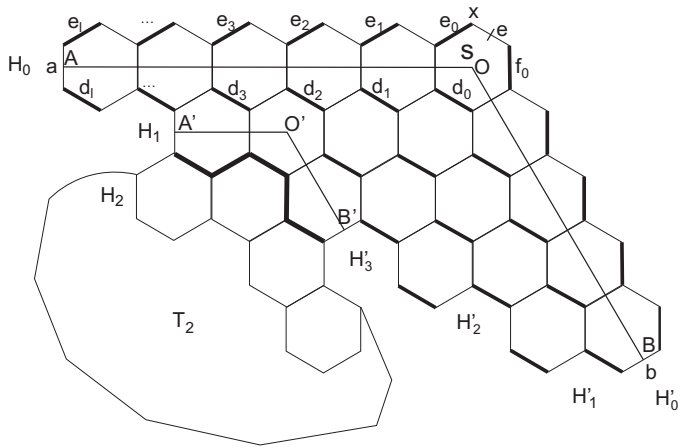}
\caption{Illustration for the proof of Theorem \ref{af2} ($m=1,m'=3$)}\label{Case1}
\end{figure}

\noindent {\bf Claim 3}. It is impossible that $H$ has not only hexagons adjacent above to $H_0$  but also hexagons adjacent right to $H_0'$.

By Claim 3 we may suppose that $H$ has no hexagons adjacent above to $H_0$. Let $e_1,e_2,\ldots, e_l$ denote a series of edges in $H_0$ parallel to $e_0$ and above $OA$, $d_1,d_2,\ldots, d_l$ denote a series of edges in $H_0$ parallel to $d_0$ and below $OA$ (see Fig. \ref{Case1}).  Hence $e_1, e_2,\ldots, e_l$ are anti-forced by $e$ in turn and thus belong to $M$.

Let $H_1$ be the graph consisting of the hexagons adjacent to $H_0$ and below it. If $d_l$ is a boundary edge of $H$, then $d_l,\ldots,d_1,d_0$ are further anti-forced by $e$ and thus belong to $M$. So $H_1$ is a linear chain with an end hexagon in $H_0'$, and thus $H_1$ has at most many hexagons as $H_0$. Otherwise, by Lemma \ref{constant} we have that  some vertical edges in $H_1$ are fixed single edges, contradicting that $H$ is normal. In general, for $m\ge 0$ let $H_{m+1}$ be the graph consisting of the hexagons adjacent to $H_m$ and below it. If $H_{m+1}$ has no hexagon adjacent left to the left end hexagon of $H_m$, by the same reasons as above we have that $H_{m+1}$ is a linear chain with an end hexagon in $H_0'$ and the edges in $H_{m+1}$ parallel to $d_0$ are anti-forced by $e$ and thus belong to $M$. There are two cases to be considered.

\noindent {\bf Case 1}. $H$ has no hexagons adjacent right to $H_0'$.

In this case there must be an integer $m$ such that for each $1\leq i\leq m$, $H_i$ is a linear chain with an end hexagon in $H_0'$ and  $H_i$ has at most many hexagons as $H_{i-1}$, but $H_{m+1}$ has a hexagon adjacent left to the left end hexagon of $H_m$. Otherwise $H$ is a truncated parallelogram, a contradiction. Along chain $H_0'$, similarly as rows $H_i$ we can define $H'_j$ in turn and have the similar fact: there must be an integer $m'$ such that for each $1\leq j\leq m'$, $H_j'$ is a linear chain with an end hexagon in $H_0$ and  $H_j'$ has at most many hexagons as $H_{j-1}'$, but $H_{m'+1}$ has a hexagon adjacent below to the lowest hexagon of $H_{m'}'$ (see Fig. 7).
Then $H_m$ and $H_{m'}'$ have exactly one hexagon $s'$ in common. Let $O'$ be the center of $s'$,  $A'$  the center of the most-left vertical edge of $H_m$ and  $B'$ the center of the lowest right edge of $H'_{m'}$. Hence $T_2:=H\ominus e$ is just a subhexagonal system lying in left-low side of the line $A'O'B'$. Let $T_1$ be the graph consisting of $H_0,\ldots,H_m$ and $H_0',\ldots,H_{m'}'$. It is obvious that $T_1$ is a truncated parallelogram,   $T_1$ and $T_2$ intersect at a path of odd length, and statements (i) and (ii) holds.
\begin{figure}
\centering
\includegraphics[scale=0.65]{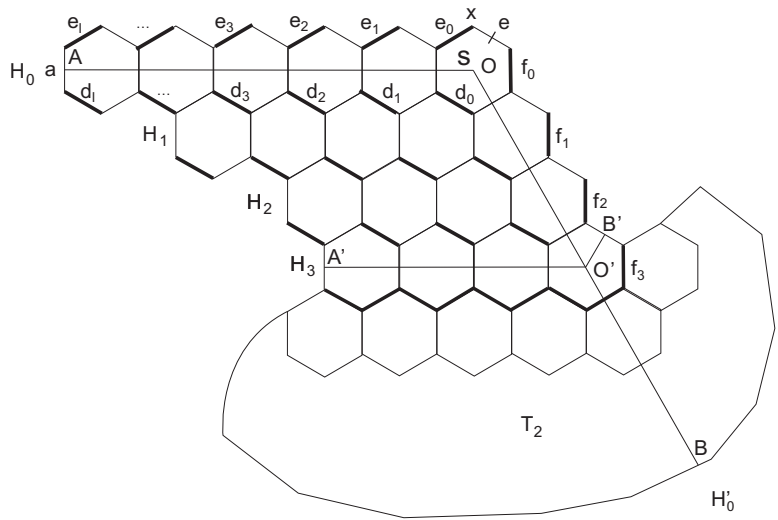}
\caption{Illustration for the proof of Theorem \ref{af2} ($m=3$)}\label{Case2}
\end{figure}

\noindent {\bf Case 2}. $H$ has hexagons adjacent right to $H_0'$.
Let $H_0''$ be the graph consisting of hexagons of $H$ adjacent right to $H_0'$. Let $m$ be the least integer such that $H_{m}$ has the right end hexagon adjacent to a hexagon of $H_0''$. Note that $m$ may be  zero. Let $f_0,f_1,\ldots,f_{m}$ be a series of vertical edges of $H'_0$ on its right side (see Fig. \ref{Case2}). Then the edges $f_0,f_1,\ldots,f_{m-1}$ are anti-forced by $e$ and thus belong to $M$. If every $H_i$ is a linear chain and $H_{i}$ has no hexagons adjacent left to the left-end hexagon of $H_{i-1}$, then $H_0''$ is a linear chain that intersect $H_0'$ at a path of odd length, so $T_2=H\ominus e$ must be a truncated parallelogram consisting of $H_0''$ and its right side.  Otherwise, by analogous arguments we have that for each $1\leq i\leq m$, $H_i$ is a linear chain with an end hexagon in $H_0'$ and  $H_i$ has at most many hexagons as $H_{i-1}$, but $H_{m+1}$ has a hexagon adjacent left to the left end hexagon of $H_m$. Let $s'$ be the right end hexagon of $H_m$,  $O'$  the center of $s'$,  $A'$  the center of the most-left vertical edge of $H_m$ and  $B'$ the center of the  edge of $s'$ adjacent above to $f_m$. Then $T_2:=H\ominus e$ just lies below $A'O'B'$, and $T_1$ consists of $H_0,H_1,\ldots, H_m$ (see Fig. \ref{Case2}).  So  the necessity is proved.

Conversely, suppose that $H$ is obtained from the construction that the theorem states. We can see that the anti-forcing edge $e$ of $T_1$ can anti-forces all double and single edges of $T_1$ except for the path $P$. That is, $H\ominus e=T_2$. Hence $af(H)\leq 2$. Since $H$ is not truncated parallelogram, $af(H)=2$.
\end{proof}

\begin{figure}
\centering
\includegraphics[scale=0.7]{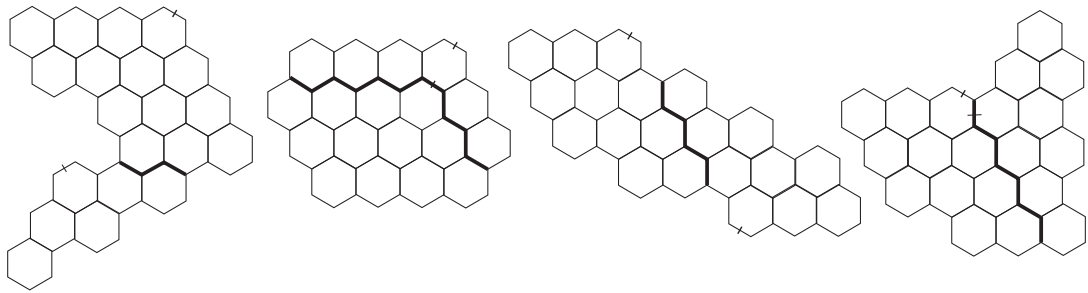}
\caption{Four examples for gluing two truncated parallelograms along a path of odd length (marked by bold line) to get hexagonal systems with the anti-forcing number 2.}\label{examples}
\end{figure}

Finally we give some examples of applying the construction of Theorem \ref{af2} as shown in Fig. \ref{examples}. The last graph has the minimum forcing number one. In fact, Zhang and Li \cite{ZL}, and Hansen and Zheng \cite{Hansen} determined hexagonal systems with a forcing edge. In hexagonal systems $H$ with $af(H)\leq 2$, we can see that in addition to such kind of graphs,    we always have that $af(H)=f(H)$. \\

\noindent{\bf\large Acknowledgments}

\noindent This work is  supported by the  National Natural Science  Foundation of China (Grant No. 11371180), and Institute of Mathematics, Academia Sinica, Taipei.

\end{document}